\documentclass[11pt]{amsart}
\usepackage{amssymb,amsthm,amsmath,epsfig,latexsym}
\usepackage{calc,times,verbatim}
\begin{document}

\newcommand{\mmbox}[1]{\mbox{${#1}$}}
\newcommand{\proj}[1]{\mmbox{{\mathbb P}^{#1}}}
\newcommand{\Cr}{C^r(\Delta)}
\newcommand{\CR}{C^r(\hat\Delta)}
\newcommand{\affine}[1]{\mmbox{{\mathbb A}^{#1}}}
\newcommand{\Ann}[1]{\mmbox{{\rm Ann}({#1})}}
\newcommand{\caps}[3]{\mmbox{{#1}_{#2} \cap \ldots \cap {#1}_{#3}}}
\newcommand{\Proj}{{\mathbb P}}
\newcommand{\N}{{\mathbb N}}
\newcommand{\Z}{{\mathbb Z}}
\newcommand{\R}{{\mathbb R}}
\newcommand{\A}{{\mathcal{A}}}
\newcommand{\Tor}{\mathop{\rm Tor}\nolimits}
\newcommand{\Ext}{\mathop{\rm Ext}\nolimits}
\newcommand{\Hom}{\mathop{\rm Hom}\nolimits}
\newcommand{\im}{\mathop{\rm Im}\nolimits}
\newcommand{\rank}{\mathop{\rm rank}\nolimits}
\newcommand{\supp}{\mathop{\rm supp}\nolimits}
\newcommand{\arrow}[1]{\stackrel{#1}{\longrightarrow}}
\newcommand{\CB}{Cayley-Bacharach}
\newcommand{\coker}{\mathop{\rm coker}\nolimits}
\sloppy
\newtheorem{defn0}{Definition}[section]
\newtheorem{prop0}[defn0]{Proposition}
\newtheorem{quest0}[defn0]{Question}
\newtheorem{thm0}[defn0]{Theorem}
\newtheorem{lem0}[defn0]{Lemma}
\newtheorem{corollary0}[defn0]{Corollary}
\newtheorem{example0}[defn0]{Example}
\newtheorem{remark0}[defn0]{Remark}

\newenvironment{defn}{\begin{defn0}}{\end{defn0}}
\newenvironment{prop}{\begin{prop0}}{\end{prop0}}
\newenvironment{quest}{\begin{quest0}}{\end{quest0}}
\newenvironment{thm}{\begin{thm0}}{\end{thm0}}
\newenvironment{lem}{\begin{lem0}}{\end{lem0}}
\newenvironment{cor}{\begin{corollary0}}{\end{corollary0}}
\newenvironment{exm}{\begin{example0}\rm}{\end{example0}}
\newenvironment{rem}{\begin{remark0}\rm}{\end{remark0}}

\newcommand{\defref}[1]{Definition~\ref{#1}}
\newcommand{\propref}[1]{Proposition~\ref{#1}}
\newcommand{\thmref}[1]{Theorem~\ref{#1}}
\newcommand{\lemref}[1]{Lemma~\ref{#1}}
\newcommand{\corref}[1]{Corollary~\ref{#1}}
\newcommand{\exref}[1]{Example~\ref{#1}}
\newcommand{\secref}[1]{Section~\ref{#1}}
\newcommand{\remref}[1]{Remark~\ref{#1}}
\newcommand{\questref}[1]{Question~\ref{#1}}

\newcommand{\std}{Gr\"{o}bner}
\newcommand{\jq}{J_{Q}}


\title[Arrangements with quadratic log vector fields]{Projective duality of arrangements with quadratic logarithmic vector fields}
\author{Stefan O. Tohaneanu}

\subjclass[2010]{Primary 52C35; Secondary: 13D02} \keywords{hyperplane arrangements, logarithmic derivation, syzygies, Jacobian ideal. \\ \indent Department of Mathematics, University of Idaho, Moscow, Idaho 83844-1103, USA, Email: tohaneanu@uidaho.edu, Phone: 208-885-6234, Fax: 208-885-5843.}

\begin{abstract}
In these notes we study hyperplane arrangements having at least one logarithmic derivation of degree two that is not a combination of degree one logarithmic derivations. It is well-known that if a hyperplane arrangement has a linear logarithmic derivation not a constant multiple of the Euler derivation, then the arrangement decomposes as the direct product of smaller arrangements. The next natural step would be to study arrangements with non-trivial quadratic logarithmic derivations. On this regard, we present a computational lemma that leads to a full classification of hyperplane arrangements of rank 3 having such a quadratic logarithmic derivation. These results come as a consequence of looking at the variety of the points dual to the hyperplanes in such special arrangements. \end{abstract}
\maketitle

\section{Introduction}

Let $\mathcal A$ be a central essential hyperplane arrangement in $V$ a vector space of dimension $k$ over $\mathbb K$ a field of characteristic zero. Let $R=Sym(V^*)=\mathbb K[x_1,\ldots,x_k]$ and fix $\ell_i\in R,i=1,\ldots,n$ the linear forms defining the hyperplanes of $\mathcal A$. After a change of coordinates, assume that $\ell_i=x_i,i=1,\ldots,k$.

A {\em logarithmic derivation} (or {\em logarithmic vector field}) of $\mathcal A$ is an element $\theta\in Der(R)$, such that $\theta(\ell_i)\in\langle \ell_i\rangle$, for all $i=1,\ldots,n$. Picking the standard basis for $Der(R)$, i.e., $\partial_1:=\partial_{x_1},\ldots,\partial_k:=\partial_{x_k}$, if $\theta$ is written as $$\theta=\sum_{i=1}^kP_i\partial_i,$$ where $P_i\in R$ are homogeneous polynomials of the same degree, then $\deg(\theta)=\deg(P_i)$. The set of logarithmic derivations forms an $R-$module, and whenever this module is free one says that the hyperplane arrangement is {\em free}.

In general, every central hyperplane arrangement has {\em the Euler derivation}: $$\theta_E=x_1\partial_1+\cdots+x_k\partial_k.$$ There exists a one-to-one correspondence between logarithmic derivations not multiples of $\theta_E$ and the first syzygies on the {\em Jacobian ideal of $\mathcal A$}, which is the ideal of $R$ generated by the (first order) partial derivatives of the defining polynomial of $\mathcal A$.\footnote{For more details about this, and in general about the theory of hyperplane arrangements, the first place to look is the landmark book of Orlik and Terao, \cite{OrTe}.} Therefore, we are interested in hyperplane arrangements that have a minimal quadratic syzygy, on its Jacobian ideal. Throughout the notes we are going to use both terminologies: ``non-trivial logarithmic derivation'' or ``minimal quadratic syzygy''.

\cite{Zi} presents interesting constructions of hyperplane arrangements with linear or quadratic syzygies. These are summed up in Proposition 8.5: if $\mathcal A$ is an essential arrangement with $e(\mathcal A')\leq 2$, for all subarrangements $\mathcal A'\subseteq \mathcal A$, then the Jacobian ideal of $\mathcal A$ has only linear or quadratic syzygies, which are combinatorially constructed. Here $e(\mathcal A')=\min_{H\in\mathcal A'}(|\mathcal A'|-|\mathcal A'|_H|)$ is the {\em excess} of $\mathcal A'$. A question comes up immediately: is it true that a free hyperplane arrangement with exponents $1$'s and $2$'s is supersolvable? The answer would be yes, if one shows that a free arrangement with exponents $1$'s and $2$'s has quadratic Orlik-Terao algebra. Then, by using \cite[Theorems 5.1 and 5.11]{DeGaTo}, one obtains supersolvability.

In these notes we do not discuss the freeness of the hyperplane arrangements we study. We are more interested in the geometry of the configuration of points that are dual to the hyperplanes of an arrangement that has a quadratic syzygy on its Jacobian ideal.

At the beginning of the next section we briefly review \cite[Proposition 4.29(3)]{OrTe} that characterizes hyperplane arrangements with a linear syzygy, and we look at this result from the projective duality view mentioned already. Next we study hyperplane arrangements with a quadratic minimal syzygy. We also obtain that the dual points lie on an interesting variety, though its description is not even close to the nice combinatorial case of the linear syzygy. Nevertheless, using this description we are able to classify up to a change of coordinates all rank 3 hyperplane arrangements having a quadratic minimal syzygy on their Jacobian ideal (Theorem \ref{main2}). We end with two questions, one addressing a simpler and shorter proof of Theorem \ref{main2}, and the other asking if it is possible to obtain a similar classification but for higher rank arrangements with a quadratic minimal logarithmic vector field.

\section{Arrangements with low degree logarithmic derivations}

\subsection{Linear logarithmic derivations}

Dropping the freeness condition which is not necessary in our study, \cite[Proposition 4.29(3)]{OrTe} shows the following: If $\mathcal A$ is an arrangement with $e_1$ linearly independent degree 1 logarithmic derivations (including $\theta_E$), then $\mathcal A$ is a direct product of $e_1$ irreducible arrangements.

One can obtain the same result, with a different interpretation of $e_1$, by studying the points dual to the hyperplanes of $\mathcal A$ in the following manner. Keeping the notations from the beginning of Introduction, let us assume that $\mathcal A$ has a linear logarithmic derivation, not a constant multiple of $\theta_E$: $$\theta=L_1\partial_1+\cdots+L_k\partial_k,$$ where $L_j$ are some linear forms in $R$. Because $\theta(x_i)=a_ix_i, i=1,\ldots,k$ for some constants $a_i\in\mathbb K$, then $$L_i=a_ix_i,i=1,\ldots,k,$$ and not all $a_i$'s are equal to each-other (otherwise we would get a constant multiple of $\theta_E$).

For $i\geq k+1$, suppose $\ell_i=p_{1,i}x_1+\cdots+p_{k,i}x_k,p_{j,i}\in\mathbb K$. The logarithmic condition $\theta(\ell_i)=\lambda_i\ell_i,i\geq k+1,\lambda_i\in\mathbb K$ translates into $$\left[\begin{array}{ccc}
a_1& &0\\
 &\ddots&\\
 0& &a_k
 \end{array}\right]\cdot \left[\begin{array}{c}
 p_{1,i}\\
 \vdots\\
 p_{k,i}
 \end{array}\right]=\lambda_i\left[\begin{array}{c}
 p_{1,i}\\
 \vdots\\
 p_{k,i}
 \end{array}\right].$$ Therefore the points in $\mathbb P^{k-1}$ dual to the hyperplanes $\ell_i$ sit on the scheme with defining ideal $I$ generated by the $2\times 2$ minors of the matrix $$\left[\begin{array}{cccc}
a_1x_1&a_2x_2&\cdots&a_kx_k\\
x_1&x_2&\cdots&x_k
 \end{array}\right].$$

 Obviously $$I=\langle\{(a_i-a_j)x_ix_j:i\neq j\}\rangle,$$ and this is the edge (graph) ideal of a complete multipartite simple graph on vertices $1,\ldots,k$; two vertices $u$ and $v$ belong to the same partition iff $a_u=a_v$.

 For a simple graph $G$, a minimal vertex cover is a subset of vertices of $G$, minimal under inclusion, such that every edge of $G$ has at least one vertex in this subset. By \cite{Vi}, since $I$ is the edge ideal of a simple graph $G$ (complete multipartite), all the minimal primes of $I$ are generated by subsets of variables corresponding to the minimal vertex covers of $G$.\footnote{Greg Burnham, an REU student of Jessica Sidman, attributes this well known result to Rafael Villarreal, so we decided to use the same citation.}  Also, since $I$ is generated by square-free monomials, it must be a radical ideal, hence it is equal to the intersection of its minimal primes.

It is not difficult to show that if $G$ is a complete multipartite graph with partition $P_1,\ldots,P_s$, then the minimal vertex covers of $G$ are $V(G)-P_i,i=1,\ldots,s$. So $$I=I(G)=\cap_{i=1}^s\langle\{x_v:v\in V(G)-P_i\}\rangle.$$ The points dual to the hyperplanes of $\mathcal A$ are in the zero set (the variety) of $I$. If $[p_1,\ldots,p_k]\in V(I)$, then there is $1\leq j\leq s$ with $p_v=0$, for all $v\in V(G)-P_j$. Then the linear form dual to this point belongs to $\mathbb K[x_v, v\in P_j]$, so it defines a hyperplane in $\mathbb P^{|P_j|-1}$.

Since $\mathcal A$ has full rank, each component must contain at least one of these points, and therefore we can group the linear forms accordingly to the components their dual points belong to. So $\mathcal A=\mathcal A_1\times\cdots\times\mathcal A_s$ where $\mathcal A_i\subset \mathbb P^{|P_i|-1},i=1,\ldots,s$ and $P_1,\ldots,P_s$ is the partition of the complete multipartite graph we have seen above.

\subsection{Quadratic logarithmic derivations}

In this subsection we consider hyperplane arrangements with quadratic logarithmic derivations, not a linear combination of linear logarithmic derivations. In other words, the Jacobian ideal has a minimal quadratic syzygy.

Let $\mathcal A$ be as before, with $$\ell_i=x_i, i=1,\ldots,k$$ and $$\ell_j=p_{1,j}x_1+\cdots+p_{k,j}x_k, j\geq k+1.$$

Let $\theta = Q_1\partial_1+\cdots+Q_k\partial_k$ be a quadratic logarithmic derivation, $Q_i\in R:=\mathbb K[x_1,\ldots,x_k]$ quadratic homogeneous polynomials, assumed to have no common divisor.

For $i=1,\ldots,k$, since $\theta(x_i)=L_ix_i$ for linear form $$L_i=b_{1,i}x_1+\cdots+b_{k,i}x_k, b_{u,i}\in\mathbb K,$$ then $Q_i=L_ix_i, i=1,\ldots,k$.

Similarly to the linear syzygy case, we will analyse the dual points to each hyperplane in $\mathcal A$, and in fact the configuration of these points if $\mathcal A$ has a quadratic logarithmic derivation. The next result gives the first insights into this regard.

\begin{lem}\label{lemma_quadratic} Let $\mathcal A$ be a hyperplane arrangement with a quadratic logarithmic derivation. If $V(\ell_j)\in\mathcal A$, where $\ell_j=p_{1,j}x_1+\cdots+p_{k,j}x_k$ with $p_{u,j},p_{v,j}\neq 0$, then $$[p_{1,j},\ldots,p_{k,j}]\in V(I_{u,v}),$$ where $I_{u,v}$ is the ideal of $R$ generated by the following $k-1$ elements: $$x_u(b_{v,u}-b_{v,v})+x_v(b_{u,v}-b_{u,u}),$$ and $$x_ux_v(b_{w,u}-b_{w,v})+x_vx_w(b_{u,w}-b_{u,u})-x_ux_w(b_{v,w}-b_{v,v}), w\neq u,v.$$
\end{lem}
\begin{proof} Suppose $p_{1,j}\neq 0$ and $p_{2,j}\neq 0$.

We have that $\theta(\ell_j)=\ell_j(A_{1,j}x_1+\cdots+A_{k,j}x_k), A_{i,j}\in\mathbb K$, leading to $$L_1x_1p_{1,j}+ \cdots+L_kx_kp_{k,j}=(p_{1,j}x_1+\cdots+p_{k,j}x_k)(A_{1,j}x_1+\cdots+A_{k,j}x_k).$$

Identifying coefficients one obtains the following equations relevant to our calculations:

\begin{eqnarray}
p_{1,j}(b_{1,1}-A_{1,j})&=&0\nonumber\\
p_{2,j}(b_{2,2}-A_{2,j})&=&0\nonumber
\end{eqnarray} and

\begin{eqnarray}
p_{1,j}(b_{2,1}-A_{2,j})+p_{2,j}(b_{1,2}-A_{1,j})&=&0\nonumber\\
p_{1,j}(b_{u,1}-A_{u,j})+p_{u,j}(b_{1,u}-A_{1,j})&=&0, u\geq 3\nonumber\\
p_{2,j}(b_{u,2}-A_{u,j})+p_{u,j}(b_{2,u}-A_{2,j})&=&0, u\geq 3.\nonumber
\end{eqnarray}

Since $p_{1,j},p_{2,j}\neq 0$, we have $A_{1,j}=b_{1,1}$ and $A_{2,j}=b_{2,2}$, from the first two equations, and from the second group of equations we have $$p_{1,j}(b_{2,1}-b_{2,2})+p_{2,j}(b_{1,2}-b_{1,1})=0$$ and for all $u\geq 3$
\begin{eqnarray}
A_{u,j}&=& b_{u,1}+\frac{p_{u,j}}{p_{1,j}}(b_{1,u}-b_{1,1})\nonumber\\
    &=& b_{u,2}+\frac{p_{u,j}}{p_{2,j}}(b_{2,u}-b_{2,2}).\nonumber
\end{eqnarray}

From these one obtains that the dual point to the line $\ell_j=0$, belongs to the ideal of $R$ generated by

$$x_1(b_{2,1}-b_{2,2})+x_2(b_{1,2}-b_{1,1})$$ and $$\{x_1x_2(b_{u,1}-b_{u,2})+x_2x_u(b_{1,u}-b_{1,1})-x_1x_u (b_{2,u}-b_{2,2})\}_{u\geq 3}.$$
\end{proof}

If $j\geq k+1$, $\ell_j$ has at least two non-zero coefficients. If $1\leq i\leq k$, $\ell_i=x_i$ and the dual point to this hyperplane belongs to $V(I_{u,v})$ for any $u,v\neq i$. Summing up we obtain the following:

\begin{cor}\label{char} If a hyperplane arrangement $\mathcal A$ has a quadratic logarithmic derivation then the points dual to the hyperplanes of $\mathcal A$ lie on the variety $\displaystyle\bigcup_{1\leq u<v\leq k}V(I_{u,v})$, where each ideal $I_{u,v}$ is defined as in Lemma \ref{lemma_quadratic}.
\end{cor}

\subsection{The case of line arrangements in $\mathbb P^2$} In this subsection we classify the line arrangements in $\mathbb P^2$ having a minimal quadratic syzygy on its Jacobian ideal. To differentiate from the previous Subsection 2.1, we assume further that this Jacobian ideal does not have a linear syzygy.

In what follows $\mathcal A$ has defining linear forms $\ell_1=x,\ell_2=y,\ell_3=z,$ and $\ell_i=\alpha_ix+\beta_iy+\gamma_iz,i\geq 4$.

With the previous notations we have
\begin{eqnarray}
L_1&=&b_{1,1}x+b_{2,1}y+b_{3,1}z\nonumber\\
L_2&=&b_{1,2}x+b_{2,2}y+b_{3,2}z\nonumber\\
L_3&=&b_{1,3}x+b_{2,3}y+b_{3,3}z\nonumber
\end{eqnarray} and $$Q_1=xL_1,Q_2=yL_2,Q_3=zL_3.$$

From Corollary \ref{char}, we have the points dual to the lines, meaning $[\alpha_i,\beta_i,\gamma_i]$,  sitting on $$V(I_{xy})\cup V(I_{xz}) \cup V(I_{yz}),$$ where
\begin{tiny}
\begin{eqnarray}
I_{xy}&=&\langle x(b_{2,1}-b_{2,2})+y(b_{1,2}-b_{1,1}), xy(b_{3,1}-b_{3,2})+yz(b_{1,3}-b_{1,1})-xz(b_{2,3}- b_{2,2})\rangle\nonumber\\
I_{xz}&=&\langle x(b_{3,1}-b_{3,3})+z(b_{1,3}-b_{1,1}), xz(b_{2,1}-b_{2,3})+yz(b_{1,2}-b_{1,1})-xy(b_{3,2}- b_{3,3})\rangle\nonumber\\
I_{yz}&=&\langle y(b_{3,2}-b_{3,3})+z(b_{2,3}-b_{2,2}), yz(b_{1,2}-b_{1,3})+xz(b_{2,1}-b_{2,2})-xy(b_{3,1}- b_{3,3})\rangle.\nonumber
\end{eqnarray}
\end{tiny}

Denote \begin{eqnarray}
a_1&:=& b_{2,1}-b_{2,2}\nonumber\\
b_1&:=& b_{1,2}-b_{1,1}\nonumber\\
a_2&:=& b_{3,1}-b_{3,3}\nonumber\\
c_2&:=& b_{1,3}-b_{1,1}\nonumber\\
b_3&:=& b_{3,2}-b_{3,3}\nonumber\\
c_3&:=& b_{2,3}-b_{2,2}.\nonumber
\end{eqnarray} Then our ideals of interest become:

\begin{eqnarray}
I_{xy}&=&\langle a_1x+b_1y, y(a_2x+c_2z)-x(b_3y+c_3z)\rangle\nonumber\\
I_{xz}&=&\langle a_2x+c_2z, z(a_1x+b_1y)-x(b_3y+c_3z)\rangle\nonumber\\
I_{yz}&=&\langle b_3y+c_3z, z(a_1x+b_1y)-y(a_2x+c_2z)\rangle.\nonumber
\end{eqnarray}

\begin{lem} \label{special_case} In the assumptions of this subsection, none of the ideals $I_{xy},I_{xz},I_{yz}$ is the zero ideal.
\end{lem}
\begin{proof} If one of these ideals is the zero ideal, say $I_{xy}$, then $a_1=b_1=c_2=c_3=0$ and $a_2=b_3$. This leads to $$L_1=L_2=x+y+sz, L_3=x+y+tz,$$ with $t\neq s$ (otherwise obtaining a linear syzygy on the Jacobian ideal). Then the quadratic logarithmic derivation becomes $$\theta=x(x+y+sz)\partial_x+y(x+y+sz)\partial_y+z(x+y+tz)\partial_z.$$

Let $\ell=\alpha x+\beta y+\gamma z$ be a linear form defining a line in $\mathcal A$, but different than $\ell_1,\ell_2$ or $\ell_3$.

$\theta(\ell)=\ell(Ax+By+Cz),$ for some $A,B,C\in\mathbb K$ gives
\begin{eqnarray}
\alpha&=&\alpha A\nonumber\\
\beta&=&\beta B\nonumber\\
t\gamma&=&\gamma C\nonumber\\
\alpha+\beta&=&\alpha B+\beta A\nonumber\\
s\alpha+\gamma&=&\alpha C+\gamma A\nonumber\\
s\beta+\gamma&=&\beta C+\gamma B.\nonumber
\end{eqnarray}

If $\alpha,\beta,\gamma\neq 0$, then $A=B=1, C=t$. Fifth equation gives also $C=s$, which contradicts with $s\neq t$.

If $\alpha=0$, then $\beta, \gamma\neq 0$, otherwise we'd get $\ell_2$ or $\ell_3$. Then $B=1$ and $C=t$. The sixth equation also gives $C=s$, contradiction with $s\neq t$.

If $\gamma=0$, then $\alpha,\beta\neq 0$, giving $A=B=1, C=s$. {\em A priori} this could happen only if the defining polynomial of $\mathcal A$ is $xyz\prod(\alpha_ix+\beta_iy)$. But this is a contradiction with the setup of this subsection: our arrangements do not have a linear syzygy.
\end{proof}

\begin{thm}\label{main2} Let $\mathcal A$ be a line arrangement in $\mathbb P^2$, having a minimal quadratic syzygy on its Jacobian ideal, but not a linear syzygy. Then, up to a change of coordinates, $\mathcal A$ is one of the following three types of arrangements with defining polynomials (see also their affine pictures below):
\begin{enumerate}
  \item $F=xyz(x+y)\prod_j(y+t_jz), t_j\neq 0$.
  \item $F=xyz(x+y+z)\prod_j(y+t_jz), t_j\neq 0$.
  \item $F=xyz(x+y+z)(x+z)(y+z)$.
\end{enumerate}
\end{thm}

\begin{center}
\epsfig{file=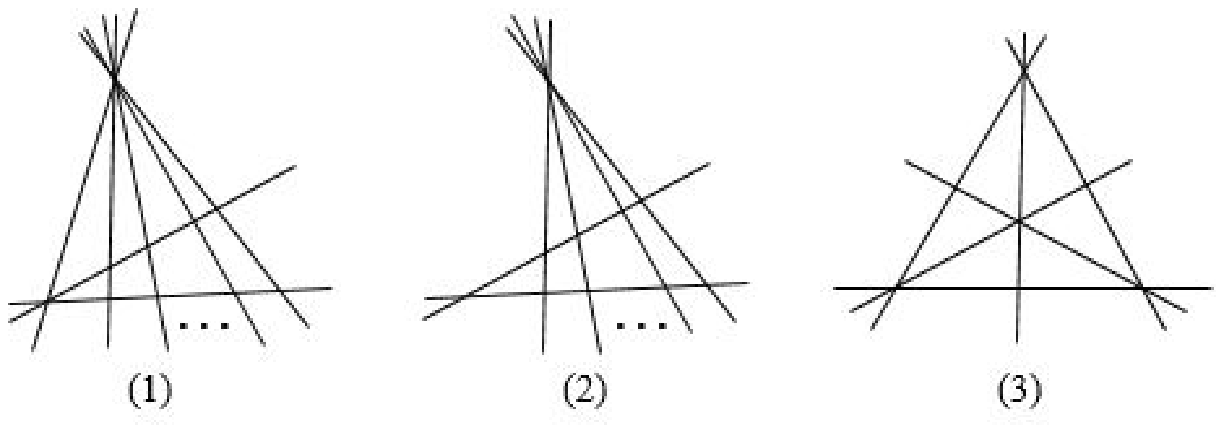,height=1.75in,width=4.5in}
\end{center}

\begin{proof}

A couple of observations are in place:

$\bullet$ The point $[0,0,1]$ (dual to $\ell_3$) is in $V(I_{xy})$, the point $[0,1,0]$ (dual to $\ell_2$) is in $V(I_{xz})$, and the point $[1,0,0]$ (dual to $\ell_1$) is in $V(I_{yz})$.

$\bullet$ If the zero locus of any of the three ideals contains 3 or more points, then the corresponding ideal will have codimension 1, and hence will be generated by the linear generator (if the coefficients of this are not zero). This comes from the fact that if the codimension of such an ideal is 2, then the zero locus will be a finite set of points, and since the ideal is generated by a linear form and a quadric, by B\'{e}zout's theorem we can have at most 2 points in this zero locus (exactly 2 if the line and the conic intersect transversally). Lemma \ref{special_case} assures that this codimension is $> 0$.

\medskip

The proof goes through several cases enforced by these two bullets.

\vskip .2in

\noindent {\bf CASE 1:} Suppose $n_1\geq 2$ dual points have the first two coordinates different than zero. Then, from Lemma \ref{lemma_quadratic} these points belong to $V(I_{xy})$. From the two bullets above, $codim(I_{xy})=1$.

\medskip

\noindent \underline{CASE 1.1}: Suppose $a_1\neq 0$.

If $b_1=0$ then these $n_1$ points will be on $V(a_1x)$, hence their first coordinate will be zero. Contradiction. So $b_1\neq 0$, and these $n_1$ points have homogeneous coordinates $[b_1,-a_1,t]$, for some $t\in\mathbb K$.

Also, $codim(I_{xy})=1$ implies $a_1x+b_1y$ divides $y(a_2x+c_2z)-x(b_3y+c_3z)$, which is true if and only if $a_2=b_3,c_3=-a_1w,c_2=b_1w$, for some $w\in \mathbb K$.

We can have $t=0$, leading to the point $[b_1,-a_1,0]$, and the corresponding dual linear form $b_1x-a_1y$; or $t\neq 0$ and in this last situation $[b_1,-a_1,t]\in V(I_{xz})\cap V(I_{yz})$ (as all the coordinates of this point are different than zero, and from Lemma \ref{lemma_quadratic}). So $a_2b_1+c_2t=0$, leading to $wt=-a_2$. Note that also $b_3(-a_1)+c_3t=0$, but by using $c_3=-a_1w$ and $b_3=a_2$, after simplifying by $a_1\neq 0$ one obtains the same $wt=-a_2$.

\medskip

CASE 1.1.1: $a_2=0$.

If $t\neq 0$, then $w=0$ and so $a_2=b_3=c_2=c_3=0$. This leads to $$I_{xy}=\langle a_1x+b_1y\rangle, I_{xz}=I_{yz}=\langle z(a_1x+b_1y)\rangle.$$

Looking at $\ell_i=\alpha_ix+\beta_iy+\gamma_iz,i\geq 4$, if $\gamma_u=0$, for some $u$, then none of the corresponding $\alpha_u$ or $\beta_u$ can be zero, as we would obtain $\ell_1$ or $\ell_2$. So the dual point $[\alpha_u,\beta_u,0]$ has the first two coordinates $\neq 0$, the setup of CASE 1. Hence $a_1\alpha_u+b_1\beta_u=0$, equivalently obtaining the linear form $b_1x-a_1y$; same situation when $t=0$.

If $\gamma_u\neq 0$, then we obtain again that the first two coordinates of the dual points of $\ell_i, i\geq 4$ must satisfy the equation $a_1x+b_1y=0$.

This leads to the only possibility of $\mathcal A$ having the defining polynomial: $F=xyz(b_1x-a_1y)\prod(b_1x-a_1y+\gamma_jz), \gamma_j\neq 0$. After an appropriate change of coordinates one gets $$F=x(x+y)yz\prod(y+\gamma_jz),\gamma_j\neq 0,$$ which is a type (1) arrangement.

\medskip

CASE 1.1.2: $a_2\neq 0$.

Then $w\neq 0$ and $t=-a_2/w\neq 0$. We are still in the situation $b_1\neq 0$, see the beginning of CASE 1.1. So $c_2=b_1w\neq 0$.

\medskip

{\em First possibility.} Suppose there exists another (dual) point, different than $[0,1,0]$ and $[b_1,-a_1,-a_2/w]$, and with the first and last coordinate different than zero. Then, by Lemma \ref{lemma_quadratic} this point belongs to $V(I_{xz})$, and hence it has homogeneous coordinates $[c_2,t',-a_2]=[b_1,t'/w,-a_2/w]$, for some $t'\in\mathbb K$.

If $t'\neq 0$, then this extra point is in $V(I_{xy})$ as well (by Lemma \ref{lemma_quadratic}). Therefore, $t'/w=-a_1$. So our extra point is not different than $[b_1,-a_1,-a_2/w]$, though we assumed that it is. This leads to the only possible for this extra dual point to be $[c_2,0,-a_2]$.

By the second bullet, $codim(I_{xz})=1$, and hence $a_2x+c_2z$ divides $z(a_1x+b_1y)-x(b_3y+c_3z)$. Since $c_2\neq 0$, one gets that $a_1=b_3$ and $b_1=a_2w',c_3=-c_2w'$, for some $w'\in\mathbb K$.

Putting everything together we have the conditions $$a_1=a_2=b_3\neq 0 \mbox{ and }b_1=a_2w',c_3=-c_2w', c_3=-a_1w, c_2=b_1w.$$

\medskip

{\em Second possibility.} Suppose in addition to the {\em First possibility} above, there is an extra dual point with the last two coordinates different than zero, and different than $[b_1,-a_1,-a_2/w]$. This extra point must have coordinates $[t'',c_3,-b_3]=[t''/w,-a_1,-a_2/w]$ (from the conditions expressed above). Similarly as before, one must have $t''=0$, and therefore this extra point is $[0,-a_1w,-a_1]$.

\medskip

\begin{enumerate}
  \item[i.] If {\em First and Second possibilities} occur, then the defining polynomial of $\mathcal A$ is of the form $xyz(c_2x-wa_1y-a_1z)(c_2x-a_1z)(-a_1wy-a_1z)$. After an appropriate change of coordinates one gets $$F=xyz(x+y+z)(x+z)(y+z),$$ which is the braid arrangement $A_3$, or type (3) in our statement.
  \item[ii.] If just the {\em First possibility} occurs, then the defining polynomial of $\mathcal A$ is $$xyz(c_2x-wa_1y-a_1z)(c_2x-a_1z),w\neq 0,$$ an arrangement of type (1) of 5 lines.
  \item[iii.] If none of the {\em possibilities} occur, then one obtains the defining polynomial of $\mathcal A$ is $$xyz(c_2x-wa_1y-a_1z),w\neq 0,$$ an arrangement of type (2) of 4 lines.
\end{enumerate}

\medskip

\noindent \underline{CASE 1.2}: Suppose $a_1=b_1=0$. Then,
\begin{eqnarray}
I_{xy}&=&\langle y(a_2x+c_2z)-x(b_3y+c_3z)\rangle\nonumber\\
I_{xz}&=&\langle a_2x+c_2z, x(b_3y+c_3z)\rangle\nonumber\\
I_{yz}&=&\langle b_3y+c_3z, y(a_2x+c_2z)\rangle.\nonumber
\end{eqnarray}

\medskip

CASE 1.2.1: Suppose one of the $n_1$ points also has the third coordinate different than zero. So this point is of the form $[\alpha,\beta,\gamma]$, with $\alpha,\beta,\gamma\neq 0$. Lemma \ref{lemma_quadratic} implies that these coordinates must satisfy also the equations
\begin{eqnarray}
a_2\alpha+c_2\gamma&=&0\nonumber\\
b_3\beta+c_3\gamma&=&0.\nonumber
\end{eqnarray}

\medskip

{\em Situation i.} If $a_2\neq 0$ and $b_3\neq 0$, then this point must be $[c_2/a_2,c_3/b_3,-1]$. Also $c_2,c_3\neq 0$.

If there is another dual point with the first and last coordinate not equal to zero, and different than this point, then, from the two bullets at the beginning of the proof, $codim(I_{xz})=1$ leading to $a_2x+c_2z$ dividing $x(b_3y+c_3z)$. But under the conditions $a_2,b_3,c_2,c_3\neq 0$, this is impossible.

For this situation we obtain $\mathcal A$ with defining polynomial $$F=xyz(c_2x/a_2+c_3y/b_3-z)\prod_j(\alpha_jx+\beta_jy),\alpha_j,\beta_j\neq 0,$$ which, after a change of coordinates is a type (2) arrangement in our statement.

\medskip

{\em Situation ii.} If $a_2\neq 0$ and $b_3=0$, then $c_3=0$ and $c_2\neq 0$. Then our ideals are $$I_{xy}=I_{yz}=\langle y(a_2x+c_2z)\rangle, I_{xz}=\langle a_2x+c_2z\rangle.$$

Looking at $\ell_i=\alpha_ix+\beta_iy+\gamma_iz,i\geq 4$, if $\beta_u=0$, for some $u$, then the corresponding $\alpha_u,\gamma_u\neq 0$, as we would obtain $\ell_1$ or $\ell_3$. So this point is of the form $[\alpha_u,0,\gamma_u]\in V(I_{xz})$; therefore $a_2\alpha_u+c_2\gamma_u=0$, and therefore obtaining $\ell_u=c_2x-a_2z$.

If $\beta_u\neq 0$, since $[\alpha_u,\beta_u,\gamma_u]\in V(I_{xy})\cup V(I_{xz})\cup V(I_{yz})$, we obtain $a_2\alpha_u+c_2\gamma_u=0$, as well.

In this situation one obtains $$F=xyz(c_2x-a_2z)\prod_j(c_2x+\beta_jy-a_2z),\beta_j\neq 0,$$ which after a change of coordinates is an arrangement of type (1) in the statement.

\medskip

{\em Situation iii.} If $a_2=0$ and $b_3\neq 0$, then $c_2=0$ and $c_3\neq 0$. This is a similar situation as {\em Situation ii.}

\medskip

{\em Situation iv.} If $a_2=b_3=0$, then $c_2=c_3=0$, and with $a_1=b_1=0$ (the setup of CASE 1.2), one obtains $L_1=L_2=L_3$, leading to $xF_x+yF_y+zF_z=0$, contradiction.

\medskip

CASE 1.2.2: Suppose all the $n_1$ points have the last coordinate equal to zero. Then, since they are points on $V(I_{xy})$, one must have $a_2=b_3$.

Also, if the linear forms different than these $n_1$ are $\ell_1=x,\ell_2=y$ and $\ell_3=z$, then $\mathcal A$ is a pencil of lines and a line at infinity, which is the case of Section 2 about arrangements with linear syzygies. Since we exclude this particular case, we can assume that there must exist a point with the first and last coordinate not zero. This extra point is in $V(I_{xz})$, so it must satisfy the equations $$a_2x+c_2z=0\mbox{ and }a_2xy+c_3xz=0.$$ So there can exist only one such extra point: $[c_2,c_3,-a_2]$. In this case, the defining polynomial looks like: $$F=xyz(x+c_3y+z)\prod_j(\alpha_jx+\beta_jy),\alpha_j,\beta_j\neq 0,$$ which after a change of coordinates is of type (1) if $c_3=0$, and it is of type (2) if $c_3\neq 0$.

\vskip .2in

\noindent {\bf CASE 2:} Suppose that we have exactly one dual point with the first two coordinates nonzero, exactly one point with the first and last coordinates nonzero, and exactly one point with the last two coordinates nonzero. Then it is not difficult to see that we obtain a type (3) arrangement.
\end{proof}

\begin{rem}\label{remark1} In the setup of Theorem \ref{main2}, using \cite[Proposition 3.6]{To} one obtains that the singular locus of $\mathcal A$ lies on a cubic curve. First observe that in all the three types presented this is indeed the case, the cubic being a union of three lines.

Second, let us consider an arrangement of 5 generic lines. The singular locus consists of 10 points which are in sufficiently general position such that there is no cubic passing through all of them; every time one requires for a cubic to pass to such a point the dimension of the space of cubics drops by one, starting with the dimension of plane cubics being equal to 10. So $\mathcal A$ cannot have as a subarrangement an arrangement of 5 generic lines.

If this second observation would give a less computational and more inspiring proof for Theorem \ref{main2}, we would be happy to see it.
\end{rem}

\vskip .2in

We end with a remark regarding a possible generalization of Theorem \ref{main2} to hyperplane arrangements in arbitrary number of variables and having a quadratic minimal syzygy (i.e., quadratic minimal logarithmic derivation).

\begin{rem}\label{remark2} The argumentation presented in Subsection 2.1 is based on the primary decomposition of a certain edge ideal. The similar ideal of interest in the case of quadratic logarithmic derivation is $I_{\mathcal A}:=\displaystyle\bigcap_{1\leq u<v\leq k}I_{u,v}$, where each ideal $I_{u,v}$ is defined as in Lemma \ref{lemma_quadratic}. It would be really interesting to be able to follow the same approach and use the primary decomposition of $I_{\mathcal A}$ in order to prove Theorem \ref{main2}.

Does the ideal $I_{\mathcal A}$ have a meaning beyond the Corollary \ref{char}?
\end{rem}

\vskip .1in

\noindent{\bf Acknowledgement} We are grateful to Torsten Hoge for pointing out to us the result from \cite{OrTe} concerning arrangements with linear syzygies, and for additional discussions and corrections. We are also thankful to the two anonymous referees for corrections and suggestions.

\renewcommand{\baselinestretch}{1.0}
\small\normalsize 

\bibliographystyle{amsalpha}

\begin{thebibliography}{10}


\bibitem{DeGaTo} G. Denham, M. Garrousian and S. Tohaneanu,
        {Modular decomposition of the Orlik-Terao algebra,}
        Annals of Combinatorics \textbf{18}(2014), 289--312.

\bibitem{OrTe} P. Orlik and H. Terao,
        Arrangements of Hyperplanes,
        Springer-Verlag, Berlin-Heidelberg-New York 1992.

\bibitem{To} S. Tohaneanu,
        {On freeness of divisors on $\mathbb P^2$,}
        Communications in Algebra \textbf{41}(2013), 2916--2932.

\bibitem{Vi} R. Villarreal,
        Combinatorial Optimization Methods in Commutative Algebra. Preprint, 2009.

\bibitem{Zi} G. Ziegler,
        {Combinatorial construction of logarithmic differential forms,}
        Advances in Math. \textbf{76}(1989), 116--154.
\end{thebibliography}

\end{document}